\documentclass[12pt]{article}

\usepackage{amssymb}
\usepackage{amsfonts}
\usepackage{amsmath}
\usepackage[usenames]{color}
\usepackage{mathrsfs}
\usepackage{amsfonts}
\usepackage{amssymb,amsmath}
\usepackage{CJK}
\usepackage{cite}
\usepackage{cases}
\usepackage{amsthm}

\def\cl{\centerline}

\def\vs{\vspace*}

\def\L{\mathcal{L}}
\def\H{\mathcal{H}}

\def\B{\mathcal{B}}
\def\Z{\mathbb{Z}}
\def\N{\mathbb{N}}
\def\A{\mathcal{A}}
\def\K{\mathcal{K}}
\def\C{\mathbb{C}}

\numberwithin{equation}{section}
\newtheorem{theo}{Theorem}[section]
\newtheorem{defi}[theo]{Definition}
\newtheorem{coro}[theo]{Corollary}
\newtheorem{lemm}[theo]{Lemma}
\newtheorem{exam}[theo]{Example}
\newtheorem{prop}[theo]{Proposition}

\newtheorem{remark}[theo]{Remark}

\begin{document}
\begin{center}
{\large\bf  Irreducible twisted Heisenberg-Virasoro modules from  tensor products}
\end{center}
\cl{Haibo Chen, Yucai Su}\footnote{\!\!\!\!\!\!\!\!\!\!Haibo Chen: School of Statistics and Mathematics,
Shanghai Lixin University of
Accounting and Finance, Shanghai 201209,  China.  e-mail: rebel1025@126.com.
Yucai Su: School of Mathematical Sciences, Tongji University, Shanghai 200092, China. e-mail:  ycsu@tongji.edu.cn.
{\it Key words}: twisted Heisenberg-Virasoro algebra,  tensor product  module,   irreducible module
\vs{4pt}\\ {\it Mathematics Subject Classification}: 17B10 $\cdot$ 17B65 $\cdot$ 17B68}

{\renewcommand{\abstractname}{}
\begin{abstract}
ABSTRACT.  In this paper,
  we  realize polynomial $\H$-modules $\Omega(\lambda,\alpha,\beta)$ from irreducible twisted Heisenberg-Virasoro modules $\A_{\alpha,\beta}$.
It follows from $\H$-modules $\Omega(\lambda,\alpha,\beta)$ and $\mathrm{Ind}(M)$ that we obtain a  class of natural non-weight tensor product  modules
$\big(\bigotimes_{i=1}^m\Omega(\lambda_i,\alpha_i,\beta_i)\big)\otimes \mathrm{Ind}(M)$.
Then we give the necessary and sufficient conditions
under which these modules are irreducible and  isomorphic, and also give that the irreducible modules
in this class  are new.

\end{abstract}

\section{Introduction}
It is well-known that the {\it twisted Heisenberg-Virasoro algebra}  $\H$ \cite{ADK} is the universal central extension of the
Lie algebra $\bar\H$
of differential operators of order
at most one on the Laurent polynomial algebra $\C[t,t^{-1}]$, where $$\bar\H:=\left\{f(t)\frac{d}{dt}+g(t)\,\Big|\, f(t),g(t)\in\C[t,t^{-1}]\right\}.$$
More precisely,
$\H$    is
 an infinite dimensional Lie algebra
with  $\C$-basis  $\{L_m,I_m,C_i
\mid  m\in \Z,\,i=1,2,3\}$ subject to  the Lie bracket as follows:
\begin{equation*}
\aligned
&[L_m,L_n]= (n-m)L_{m+n}+\delta_{m+n,0}\frac{m^{3}-m}{12}C_1,\\&
 [L_m,I_n]=n I_{m+n}+\delta_{m+n,0}(m^{2}+m)C_2,\\&
[I_m,I_n]=n\delta_{m+n,0}C_3,
\\&
 [\H,C_1]=[\H,C_2]= [\H,C_3]=0.
\endaligned
\end{equation*}
It is clear that the subspaces spanned by  $\{I_m, C_3 \mid  0\neq m\in\Z\}$   and
  by  $\{L_m, C_1 \mid m\in\Z\}$ are respectively the  Heisenberg algebra and the Virasoro algebra.
  Notice that  the center of $\H$ is spanned by  $\{I_0,C_i\mid i=1,2,3\}$.

The representation theory on the twisted Heisenberg-Virasoro algebra has attracted a lot of attention from mathematicians and physicists.
The theory of weight twisted Heisenberg-Virasoro modules with finite-dimensional weight spaces is fairly well-developed (see \cite{SS07,LJ,LZ1}).
 While  weight modules with an infinite dimensional weight spaces were also studied (see \cite{CHS16,R}).
In the last few years, various families of non-weight irreducible  twisted Heisenberg-Virasoro modules were investigated
(see, e.g., \cite{CG,CG2,HCS,CHSY,CHS16,LWZ}). These are basically various versions of Whittaker modules and $\mathcal{U}(\C L_0)$-free modules constructed using different tricks.  However, the theory of  representation over the  twisted Heisenberg-Virasoro algebra  is far more from being well-developed.

In the present paper, we  construct a  class of non-weight $\H$-modules by taking tensor products of a finite number of irreducible $\H$-modules $\Omega(\lambda,\alpha,\beta)$  with  irreducible  $\H$-modules $\mathrm{Ind}(M)$.
At the same time,  inspired by  \cite{LZK},  a class of $\H$-modules $\A_{\alpha,\beta}$ are given. Then many interesting examples for such irreducible twisted Heisenberg-Virasoro modules with different features are provided.
 In particular, a class of irreducible polynomial modules $\Omega(\lambda,\alpha,\beta)$ over the twisted Heisenberg-Virasoro algebra are defined.

We briefly give a summary of the paper below.
 In Sections $2$ and $3$, we recall some known results and  construct  a class of modules $\A_{\alpha,\beta}$ over the twisted Heisenberg-Virasoro algebra.
  In Section $4$, we determine  the necessary and sufficient conditions for $\big(\bigotimes_{i=1}^m\Omega(\lambda_i,\alpha_i,\beta_i)\big)\otimes \mathrm{Ind}(M)$  to be irreducible.
  In Section $5$, we present the necessary and sufficient conditions for  $\H$-modules $\big(\bigotimes_{i=1}^m\Omega(\lambda_i,\alpha_i,\beta_i)\big)\otimes \mathrm{Ind}(M)$ to be isomorphic.
Finally, we prove  $\H$-modules $\big(\bigotimes_{i=1}^m\Omega(\lambda_i,\alpha_i,\beta_i)\big)\otimes \mathrm{Ind}(M)$ are new.

Throughout this paper, we respectively denote by $\C,\C^*,\Z,\Z_+$ and $\N$    the sets of complex numbers,  nonzero complex numbers, integers,  nonnegative integers  and positive integers.
 All vector spaces are assumed to be over $\C$.

\section{Some known results}
In this section,
we  recall some definitions and known results.

Let $\C[t]$ be the (associative) polynomial algebra. Denote $\partial:=t\frac{d}{dt}$.  We know that $\partial t^n=t^n(\partial+n)$ for $n\in\Z$. It is clear that the associative algebra $\widetilde{\mathcal{A}}=\C[t,\partial]$  is a proper subalgebra of the rank $1$ Weyl algebra $\C[t,\frac{d}{dt}]$. We note that $\widetilde{\mathcal{A}}$ is the universal enveloping algebra of the $2$-dimensional solvable Lie algebra $\mathfrak{a}_1=\C L_0\oplus\C L_1$, which subjects to $[L_0,L_1]=L_1$.  Let $\mathcal{K}=\C[t, t^{-1}, \partial]$ be the Laurent polynomial differential operator algebra.

 \begin{defi}
Let $\mathcal{P}$ be an associative or Lie algebra over $\C$ and $\mathcal{Q}$ be a subspace of $\mathcal{P}$. A module $V$ over $\mathcal{P}$ is called {\it $\mathcal{Q}$-torsion} if there exists a nonzero $q\in\mathcal{Q}$ such that $qv=0$ for some nonzero $v\in V$; otherwise $V$ is called {\it $\mathcal{Q}$-torsion-free}.
\end{defi}

Let us recall some results about irreducible modules over the associative algebra $\mathcal{K}$.
\begin{lemm}{\rm \cite{LZK}}
Let $V$ be any $\C[t]$-torsion-free irreducible module over the associative algebra $\widetilde{\mathcal{A}}$. Then  $V$ can be extended into a module over the associative algebra $\mathcal{K}=\C[t,t^{-1}, \frac{d}{dt}]$, i.e., the action of $\widetilde{\mathcal{A}}$ on $V$ is a restriction of an irreducible $\C[t,t^{-1}]$-torsion-free $\mathcal{K}$-module.
\end{lemm}
\begin{lemm}\label{11} {\rm \cite{LZK}}
Let $\mu$ be an irreducible element in the associative algebra $\C(t)[\partial]$. Then
$$\mathcal{K}/(\mathcal{K}\cap\big(\C(t)[\partial]\mu)\big)$$
is a $\C[t,t^{-1}]$-torsion-free irreducible module over the associative algebra $\mathcal{K}$. Moreover any $\C[t,t^{-1}]$-torsion-free irreducible module over the associative algebra $\mathcal{K}$ can be obtained in this way.
\end{lemm}

For any $\lambda\in\C^*$,  define a $\mathcal{K}$-module structure on the space $\Omega(\lambda)=\C [\partial]$, the polynomial algebra in $\partial$, by
$$t^m\partial^n=\lambda^m(\partial-m)^n,\ \partial\partial^n=\partial^{n+1}$$
for all $n\in\Z_+,m\in\Z$. Then $\Omega(\lambda)$ is an irreducible module over the associative algebra $\mathcal{K}$ for any $\lambda\in\C^*$ (see \cite{LZK}).
\begin{lemm}\label{22} {\rm \cite{LZK}}
Let $V$ be an irreducible module over the associative algebra $\mathcal{K}$ on which $\C[t,t^{-1}]$ is torsion. Then $V\cong \Omega(\lambda)$ for some $\lambda\in\C^*$.
\end{lemm}
Combining Lemmas \ref{11} and \ref{22},  a classification for all irreducible modules over the associative algebra $\mathcal{K}$ are obtained.

Now let us recall a large class of irreducible $\H$-modules,
which includes  the known irreducible modules such as   highest weight modules and Whittaker modules.
For any $e\in\Z_+$, denote by $\H_e$ the subalgebra
$$\mbox{$\sum\limits_{m\in\Z_+}$}(\C L_{m}\oplus\C  I_{m-e})\oplus\C C_1\oplus\C C_2\oplus\C C_3.$$
Take $M(c_0,c_1,c_2,c_3)$ to be an irreducible $\H_e$-module such that  $I_0,C_1,C_2$ and $C_3$ act on it as scalars $c_0,c_1,c_2,c_3$ respectively. For convenience, we   denote $M(c_0,c_1,c_2,c_3)$ by $M$ and form the induced $\H$-module
\begin{equation}\label{ind2.2}
\mathrm{Ind}(M):=\mathcal{U}(\H)\otimes_{\mathcal{U}(\H_{e})}M.
\end{equation}

\begin{theo}\label{th2.1}{\rm \cite{CG}}
Let $e\in\Z_+$ and $M$ be a simple $\H_e$-module with $c_3=0$.
Assume there exists $k\in\Z_+$ such that
\begin{itemize}\lineskip0pt\parskip-1pt\item[{\rm (1)}]
$
\left\{\begin{array}{llll}\mbox{the\ action\ of}\   I_{k}    \ \mbox{on}\   M \ \mbox{is\ injective}\ \quad  &\mbox{if \ }k\neq 0,\\[4pt]
c_0+(n-1)c_2\neq0\quad \mathrm{for\ all}\ n\in\Z\setminus\{0\} &\mbox{if \ }k=0,
\end{array}\right.
$
\item[{\rm (2)}]   $I_nM=L_mM=0$ for  all $n>k$ and $m>k+e$.
\end{itemize}
Then
\begin{itemize}\lineskip0pt\parskip-1pt
\item[{\rm (i)}]
$\mathrm{Ind}(M)$ is a simple $\H$-module$;$
\item[{\rm (ii)}]
the actions of $I_n,L_m$ on $\mathrm{Ind}(M)$ for all $n>k$ and $m>k+e$ are  locally nilpotent.
\end{itemize}
\end{theo}

 The following result will be used in the following (see \cite{LZ0}).
\begin{prop}\label{pro2.3}
Let $P$ be a vector space over $\C$ and $P_1$  a subspace of $P$. Suppose
that $\mu_1,\mu_2,\ldots,\mu_s\in\C^*$ are pairwise distinct, $v_{i,j}\in P$
and $f_{i,j}(t)\in\C[t]$ with $\mathrm{deg}\,f_{i,j}(t)=j$ for $i=1,2,\ldots,s;j=0,1,2,\ldots,k.$
If $$\sum_{i=1}^{s}\sum_{j=0}^{k}\mu_i^mf_{i,j}(m)v_{i,j}\in P_1\quad{ \it for}\ K< m\in\Z\ (K \ {\it any\ fixed\ element\ in}\  \Z\cup\{-\infty\})$$
then $v_{i,j}\in P_1$ for all $i,j$.
\end{prop}

\section{Realize $\H$-module $\Omega(\lambda,\alpha,\beta)$}
Let $\A$ be an irreducible module  over the associative algebra $\K$. For any $\alpha,\beta\in\C$, we define   the action
of $\H$ on $\A$ as follows
\begin{eqnarray}
 &&\label{LI2.1}  L_mv=\big(t^{m}\partial+m\alpha t^m\big)v,\
 I_mv=\beta t^mv,\ C_iv=0
 \end{eqnarray}
 for $i\in\{1,2,3\}, m\in\Z,v\in \A$.
Denote the above action by $\A_{\alpha,\beta}$.
\begin{prop}
 For any  $\alpha,\beta\in\C$, we obtain that
 $\A_{\alpha,\beta}$
is an $\H$-module under the action given in \eqref{LI2.1}.
\end{prop}
\begin{proof}
It follows from   \eqref{LI2.1}   that  we have
\begin{eqnarray*}
&&(L_mI_n-I_nL_m)v
=\beta\big(t^{m}\partial+ m\alpha t^m\big)t^nv
-\beta t^n\big(t^{m}\partial+ m\alpha t^m\big)v
=nI_{m+n}v.
\end{eqnarray*}
That is,   $L_mI_n-I_nL_m=nI_{n+m}$ holds on  $\A_{\alpha,\beta}$. By \cite{LZK},  $L_mL_n-L_nL_m=(n-m)L_{n+m}$ holds on  $\A_{\alpha,\beta}$. Finally,  the relation   $I_mI_n-I_nI_m=0$  on $\A_{\alpha,\beta}$ is trivial.
 Thus, we obtain that $\A_{\alpha,\beta}$
is an $\H$-module.
\end{proof}
Now we recall the  necessary and sufficient conditions for $\A_{\alpha,\beta}$   to be irreducible (see \cite{LZK}).
\begin{theo}
Let $\alpha,\beta\in\C$ and $\A$ be   an irreducible module   over the association algebra $\K$.
Then  $\A_{\alpha,\beta}$ as an irreducible  $\H$-modules if and only if one of the following holds
 \begin{itemize}\lineskip0pt\parskip-1pt
\item[\rm (1)]  $\alpha\notin\{0,1\}$ or $\beta\neq0$.
\item[{\rm (2)}] $\alpha=1,\beta=0$ and  $\partial\A=\A$.

\item[{\rm (3)}] $\alpha=\beta=0$ and  $\A$ is  not isomorphic to the natural $\K$ module $\C[t,t^{-1}]$.
 \end{itemize}
\end{theo}
The isomorphism results for modules $\A_{\alpha,\beta}$ as follows.
\begin{theo}
Let $\alpha_1, \alpha_2, \beta_1,\beta_2\in \C$ and $\A,\mathcal{B}$ be irreducible modules over the associative algebra $\K$.
 Then $\A_{\alpha_1, \beta_1}\cong\B_{\alpha_2, \beta_2}$ as $\H$-modules if and only if one of the following holds
 \begin{itemize}\lineskip0pt\parskip-1pt
\item[\rm (1)] $\A\cong\B$ as $\K$-modules, $\alpha_1=\alpha_2$ and $\beta_1=\beta_2$.
\item[{\rm (2)}] $\A\cong\B$ as $\K$-modules, $\alpha_1=1,\alpha_2=0,\beta_1=\beta_2=0$ and $\partial\A=\A$.

\item[{\rm (3)}] $\A\cong\B$ as $\K$-modules,  $\alpha_1=0,\alpha_2=1,\beta_1=\beta_2=0$ and $\partial\B=\B$.

 \end{itemize}
\end{theo}
\begin{proof}
{\rm (1)}
The sufficiency of the conditions is clear. Now suppose that
$\varphi:\mathcal{A}_{\alpha_1,\beta_1}\rightarrow \mathcal{A}_{\alpha_2,\beta_2}$ is an $\mathcal{H}$-module isomorphism.
 For any $v\in \mathcal{A}$, we have $\varphi(I_0v)=I_0\varphi(v)$, which gives $\beta_1=\beta_2$.
 In particular, $\beta_1\neq0$.  We note that $\varphi(I_mv)=I_m\varphi(v)$, which implies
  \begin{equation}\label{32}
  \varphi(t^mv)=t^m\varphi(v)
  \end{equation} for any $m\in\Z$.
 From  $\varphi(L_0^mv)=L_0^m\varphi(v)$, we have
  \begin{equation}\label{33}
  \varphi(\partial^mv)=\partial^m\varphi(v)
  \end{equation} for $m\in\Z$.
  Combining \eqref{32} and \eqref{33}, we obtain $\mathcal{A}\cong\mathcal{B}$ as $\mathcal{K}$-modules.
  From \eqref{32} and \eqref{33}, it is easy to get
 $$0=\varphi(L_mv)-L_m\varphi(v)=\varphi\big((t^{m}\partial+m\alpha_1 t^m)v\big)-(t^{m}\partial+m\alpha_2 t^m)\partial\varphi(v)=m(\alpha_1-\alpha_2)t^{m}\varphi(v).$$
 Then $\alpha_1=\alpha_2$. If $\beta_1=0$, these modules reduce to the Virasoro modules (see \cite{LZK}). This is {\rm (1)}.

 By the \cite[Theorem 12]{LZK}, we get  {\rm (2)} and  {\rm (3)}.
\end{proof}

Now we realize $\H$-modules $\Omega(\lambda,\alpha,\beta)$ from $\mathcal{A}_{\alpha,\beta}$.
Let $\lambda\in\C^*$ and $\alpha,\beta\in\C$.
Then we get the irreducible $\K$-module $\Omega(\lambda)$, which has a basis $\{\partial^k:k\in\Z_+\}$, and the $\K$-actions  are given by
$$t^m\cdot \partial^n=\lambda^m(\partial-m)^n,\ \partial\cdot \partial^m=\partial^{m+1}\quad \mathrm{for}\ m\in\Z, n\in\Z_+.$$
According to  \eqref{LI2.1} we have $\H$-modules $\Omega(\lambda,\alpha,\beta)=\C[\partial]$ with the action:
\begin{eqnarray*}
  L_m \partial^n=\lambda^m\big(\partial+m(\alpha-1)\big)(\partial-m)^{n},\
 I_m \partial^n= \lambda^m\beta (\partial-m)^{n}\quad {\rm for}\  m\in\Z,n\in\Z_+.
 \end{eqnarray*}
Then $\Omega(\lambda,\alpha,\beta)$ is irreducible
 if and only if $\alpha\neq1$ or $\beta\neq0$ (see \cite{CG2}).
In the following sections, we will consider a class of tensor product $\H$-modules related to  $\Omega(\lambda,\alpha,\beta)$.

Now we describe some other examples about irreducible $\H$-modules $\A_{\alpha,\beta}$,
such as intermediate series modules, degree two modules and degree $n$ modules.

\begin{exam}\label{ex5.1}
Let $\gamma\in\C[t,t^{-1}],\beta\in\C$ and $\mu=\partial-\gamma$ in Lemma \ref{11}.
Then we obtain the irreducible $\K$-module
$$\A=\K/\big(\K\cap(\C(t)[\partial]\mu)\big)=\K/(\K\mu)$$
with a basis $\{t^k: k\in\Z\}$.
We see that the $\K$-actions on $\A$ are given by
$$\partial\cdot t^n=t^n(\gamma+n),\ t^m\cdot t^n=t^{m+n}\quad \mathrm{for}\  m,n\in\Z.$$
It follows from  \eqref{LI2.1} that we get $\H$-modules $\A_{\gamma,\alpha,\beta}=\C[t,t^{-1}]$ with the action:
\begin{eqnarray*}
  L_m t^n=(\gamma+n+m\alpha)t^{m+n},\
 I_m t^n=\beta t^{m+n}\quad {\rm for}\  m,n\in\Z.
 \end{eqnarray*}
 If $\gamma\in \C\setminus\Z$ or $\alpha\notin\{0,1\}$ or $\beta\neq0$, then $\A_{\gamma,\alpha,\beta}$ is an irreducible $\H$-module (see \cite{LJ,LGZ}).
In particular, $\gamma\in\C$ these modules $\A_{\gamma,\alpha,\beta}$ are the intermediate series modules of $\H$ (see \cite{LJ,KS}).
\end{exam}

Some degree two irreducible elements in $\C(t)[\partial]$ were first constructed in  \cite{LZK}.
\begin{exam}
Let $f(t)\in \C[t,t^{-1}]$ be such that $\partial^2-f(t)$ is irreducible in $\C(t)[\partial]$. Take $\mu=\partial^2-f(t)$ in Lemma \ref{11}. Then one obtain the irreducible $\K$-module
$$\A=\K/\big(\K\cap(\C(t)[\partial]\mu)\big)=\K/(\K\mu),$$
which has a basis $\{t^k, t^k\partial: k\in\Z\}$. The $\K$-actions   on $\A$ are given by
\begin{eqnarray*}
  &&t^m\cdot t^n=t^{m+n},\  t^m\cdot (t^n\partial)=t^{m+n}\partial,\\
 &&\partial\cdot t^n=t^{n}(\partial+n),\ \partial\cdot (t^n\partial)=t^{n}(f(t)+n\partial),
 \end{eqnarray*}
where $m,n\in\Z$.
From \eqref{LI2.1}, for $\alpha\neq1$ or $\beta\neq0$, we have irreducible $\H$-modules $\A_{\alpha,\beta}=\C[t,t^{-1}]\oplus\C[t,t^{-1}]\partial$ with the action:
\begin{eqnarray*}
  &&L_m\cdot t^n=t^{m+n}(n+m\alpha+\partial),\  L_m\cdot (t^n\partial)=t^{m+n}(f(t)+m\alpha+n\partial),\\
 &&I_m\cdot t^n=\beta t^{m+n},\  I_m\cdot (t^n\partial)=\beta t^{m+n}\partial.
 \end{eqnarray*}
\end{exam}

Some degree $n$ irreducible elements in $\C(t)[\partial]$ were first constructed in  \cite{LZK}.
\begin{exam}
For any $n\in\Z_+\setminus\{0\}$, let $\mu=(\frac{d}{dt})^n-t$ in Lemma  \ref{11}. Then we have the irreducible $\K$-module
$$\A=\K/(\K\cap(\C(t)[\partial]\mu))=\K/(\K\mu),$$
which has a basis $\{t^k(\frac{d}{dt})^m: k\in\Z, m=0,1, \ldots, n-1\}$. The actions of $\K=\C[t,t^{-1}][\frac{d}{dt}]$ are given by
\begin{eqnarray*}
  &&t^k\cdot (t^r(\frac{d}{dt})^m)=t^{k+r}(\frac{d}{dt})^m\quad \mathrm{for} \ k,r\in\Z,\ 0\leq m\leq n-1,\\
  && \frac{d}{dt}\cdot (t^r(\frac{d}{dt})^m)=rt^{r-1}(\frac{d}{dt})^m+t^r(\frac{d}{dt})^{m+1}\quad \mathrm{for} \ r\in\Z,\ 0\leq m< n-1,\\
 && \frac{d}{dt}\cdot (t^r(\frac{d}{dt})^{n-1})=rt^{r-1}(\frac{d}{dt})^{n-1}+t^{r+1}\quad \mathrm{for} \ r\in\Z.
 \end{eqnarray*}

Using \eqref{LI2.1}, for  $\alpha\neq1$ or $\beta\neq0$, we obtain irreducible $\H$-modules $\A_{\alpha,\beta}=\C[t,t^{-1}]\times(\Sigma_{i=0}^{n-1}\C(\frac{d}{dt})^i)$ with the action:
\begin{eqnarray*}
  &&L_k\cdot (t^r(\frac{d}{dt})^m)=(rt^{k+r}+\alpha kt^{k+r+1})(\frac{d}{dt})^m+t^{k+r+1}(\frac{d}{dt})^{m+1},\\
 &&L_k\cdot (t^r(\frac{d}{dt})^{n-1})=(rt^{k+r}+\alpha kt^{k+r+1})(\frac{d}{dt})^{n-1}+t^{k+r+2},\\
 &&I_k\cdot (t^r(\frac{d}{dt})^m)=\beta t^{k+r}(\frac{d}{dt})^m,
 \end{eqnarray*}
where  $k,r\in\Z,0\leq m< n-1$.
\end{exam}

\section{Irreducibilities}
In this section,  we will determine the irreducibility
of  $\big(\bigotimes_{i=1}^m\Omega(\lambda_i,\alpha_i,\beta_i)\big)\otimes \mathrm{Ind}(M)$.

Now we introduce some notations. Let $m\in\N,\lambda_i,\alpha_i,\beta_i\in\C$ for $i=1,2,\ldots,m$.
Denote  $\Omega(\lambda_i,\alpha_i,\beta_i)=\C[\partial_i]$. The actions of $\H$ on $\Omega(\lambda_i,\alpha_i,\beta_i)$
are
 $$L_k \partial_i^n=\lambda_i^k\big(\partial_i+k \alpha_i)(\partial_i-k)^{n},\
 I_k \partial_i^n= \lambda_i^k\beta_i (\partial_i-k)^{n},\ C_j\partial_i^n=0$$
  for $k\in\Z,n\in\Z_+,i=1,2,\ldots,m,j=1,2,3$. Then $\Omega(\lambda_i,\alpha_i,\beta_i)$ is irreducible if and only if $\alpha_i\neq0$ or $\beta_i\neq0$ for $i=1,\ldots,m.$
For convenience, we write $\bigotimes_{i=1}^m\Omega(\lambda_i,\alpha_i,\beta_i)=\C[\partial_1,\partial_2,\ldots,\partial_m]$  for $m\in\N$.

Now we consider the tensor product $\big(\bigotimes_{i=1}^m\Omega(\lambda_i,\alpha_i,\beta_i)\big)\otimes \mathrm{Ind}(M)$. Define a total order $``\prec"$ on the  subset
$$\{\partial_1^{p_1}\cdots\partial_m^{p_m}\otimes v\mid\mathbf{P}=(p_1,\ldots,p_m)\in\Z_+^{m},m\in\N,0\neq v\in \mathrm{Ind}(M)\},$$
by
\begin{eqnarray*}\partial_1^{p_1}\cdots\partial_m^{p_m}\otimes u\prec
\partial_1^{q_1}\cdots\partial_m^{q_m}\otimes v \Longleftrightarrow
\exists k\in\N\ \mathrm{such\ that} \ p_k<q_k\ \mathrm{and} \ p_n=q_n \ \mathrm{for} \ n<k.\end{eqnarray*}
Then   each non-zero element $w$ in $\big(\bigotimes_{i=1}^m\Omega(\lambda_i,\alpha_i,\beta_i)\big)\otimes \mathrm{Ind}(M)$  can be (uniquely) written as follows
$$w=\sum_{\mathbf{p}\in I}  \partial_1^{p_1}\cdots\partial_m^{p_m}\otimes v_{\mathbf{p}},$$
where $I$ is a finite subset of $\Z_+^{m}$  and the $v_{\mathbf{p}}$ are nonzero vectors of $\mathrm{Ind}(M)$. Now we define
$\mathrm{deg}(w)=(p_1,\ldots,p_m)$, where $\partial_1^{p_1}\cdots\partial_m^{p_m}\otimes v_{\mathbf{p}}$ is the term with maximal order in the sum.  Notice that $\mathrm{deg}(1\otimes v)=\mathbf{0}=(\underbrace{0,0,\ldots,0}_{m})$.

\begin{lemm}\label{lemm2}
Let $S=\{1,2\},S^\prime\subseteq S,\lambda\in\C^*,\beta_i,\alpha_j\in\C^*,\beta_j=0$ for $i\in S^\prime,j\in S\setminus S^\prime$  and $s\in\Z_+$.
 Denote $W_s$ the vector subspace of $\Omega(\lambda,\alpha_1,\beta_1)\otimes\Omega(\lambda,\alpha_2,\beta_2)$ spanned by $\{f(\partial_1)(\partial_1+\partial_2)^n\mid n\in\Z_+,0\leq\mathrm{deg}(f)\leq s\}$ or $\{(\partial_1+\partial_2)^nf(\partial_2)\mid n\in\Z_+,0\leq\mathrm{deg}(f)\leq s\}$.
  Then $W_s$ is a submodule of $\Omega(\lambda,\alpha_1,\beta_1)\otimes\Omega(\lambda,\alpha_2,\beta_2)$.
\end{lemm}
\begin{proof}
Without loss of generality, we may assume $\lambda=1.$
For any $n\in\Z_+,f(\partial_1)\in W_s,k\in\Z$, it is easy to get
\begin{eqnarray*}
 &&I_k\big(f(\partial_1)(\partial_1+\partial_2)^n\big)
=I_k\big(\sum_{i=0}^n\binom{n}{i}f(\partial_1)\partial_1^{i}\partial_2^{n-i}\big)
\\&=&\sum_{i=0}^n\binom{n}{i}\Big(\beta_1f(\partial_1-k)(\partial_1-k)^{i}\partial_2^{n-i}+\beta_2f(\partial_1)\partial_1^{i}(\partial_2-k)^{n-i}\Big)
\\&=&\big(\beta_1f(\partial_1-k)+\beta_2f(\partial_1)\big)(\partial_1+\partial_2-k)^n\in  W_s.
\end{eqnarray*}
By Theorem 9 of \cite{TZ1},  we have
$L_k\big( f(\partial_1)(\partial_1+\partial_2)^n\big)
 \in W_s.$
 By the similar method, we obtain
$L_k\big((\partial_1+\partial_2)^nf(\partial_2)\big)\in  W_s$ and $I_k\big((\partial_1+\partial_2)^nf(\partial_2)\big)\in  W_s$.
Thus, $ W_s$ is a submodule of $\Omega(\lambda,\alpha_1,\beta_1)\otimes\Omega(\lambda,\alpha_2,\beta_2)$, completing the proof.
\end{proof}

\begin{coro}\label{coro3}
Let $S=\{1,2\},S^\prime\subseteq S,\lambda\in\C^*,\beta_i,\alpha_j\in\C^*,\beta_j=0$ for $i\in S^\prime,j\in S\setminus S^\prime$  and $s\in\Z_+$.
 Assume that $W_s$ is the subspace of $\Omega(\lambda,\alpha_1,\beta_1)\otimes\Omega(\lambda,\alpha_2,\beta_2)$,
  where $W_s$ is spanned by $\{f(\partial_1)(\partial_1+\partial_2)^n\mid n\in\Z_+,0\leq\mathrm{deg}(f)\leq s\}$ or $\{(\partial_1+\partial_2)^nf(\partial_2)\mid n\in\Z_+,0\leq\mathrm{deg}(f)\leq s\}$.
 Then  $\Omega(\lambda,\alpha_1,\beta_1)\otimes\Omega(\lambda,\alpha_2,\beta_2)$
has a series of $\L$-submodules
$$ W_1\subset W_2\subset\cdots \subset W_s\subset\cdots$$ such that $W_s/W_{s-1}\cong \Omega\big(\lambda,s+\alpha_1+\alpha_2,\beta_1+\beta_2\big)$  as $\L$-module for each $s\geq1$.
\end{coro}
\begin{proof}
For $s,n\in\Z_+,k\in\Z$, it follows from Lemma \ref{lemm2} that  we check
\begin{eqnarray*}
 &&I_k\big(\partial_1^s(\partial_1+\partial_2)^n\big)
 =I_k\big(\sum_{i=0}^n\binom{n}{i}\partial_1^{i+s}\partial_2^{n-i}\big)
 \\&\equiv&\lambda^k(\beta_1+\beta_2)\partial_1^s(\partial_1+\partial_2-k)^n\quad (\mathrm{mod}\  W_{s-1}).
\end{eqnarray*}
From Corollary 10 of \cite{TZ1}, we get
\begin{eqnarray*}
 &&L_k\big(\partial_1^s(\partial_1+\partial_2)^n\big)\equiv\lambda^k\partial_1^s\big(\partial_1+\partial_2-k(s+\alpha_0+\alpha_1)\big)(\partial_1+\partial_2-k)^n\quad (\mathrm{mod}\  W_{s-1}).
\end{eqnarray*}
By the similar method, we have
$I_k\big(\partial_1^s(\partial_1+\partial_2)^n\big)
 \equiv\lambda^k(\beta_1+\beta_2)\partial_1^s(\partial_1+\partial_2-k)^n (\mathrm{mod}\  W_{s-1}).
$
and
$L_k\big(\partial_1^s(\partial_1+\partial_2)^n\big)\equiv\lambda^k\partial_1^s\big(\partial_1+\partial_2-k(s+\alpha_0+\alpha_1)\big)(\partial_1+\partial_2-k)^n (\mathrm{mod}\  W_{s-1}).$

These show  that  the $\L$-module isomorphism $ W_s/W_{s-1}\cong \Omega(\lambda,s+\alpha_1+\alpha_2,\beta_1+\beta_2).$
\end{proof}
By the similar method in Lemma 3 of \cite{TZ1}, we get the following results.
\begin{lemm}\label{lemm431}
Let $m\in\N,\lambda_i\in\C^*,\alpha_i,\beta_i\in\C$ for $i=1,2,\ldots,m$ with the
$\lambda_i$ pairwise distinct. Then $\underbrace{1\otimes\cdots\otimes1}_{m}\otimes v$ generates the $\H$-module $\big(\bigotimes_{i=1}^m\Omega(\lambda_i,\alpha_i,\beta_i)\big)\otimes\mathrm{Ind}(M)$.
\end{lemm}

Now we  are ready to prove  the irreducibility of $\H$-module $\big(\bigotimes_{i=1}^m\Omega(\lambda_i,\alpha_i,\beta_i)\big)\otimes \mathrm{Ind}(M)$.
\begin{theo}\label{th1}
Let $m\in\N,S=\{1,\ldots,m\},S^\prime\subseteq S$ and $\lambda_i\in\C^*$ for $i\in S$  with the
$\lambda_i$ pairwise distinct.
Let $\beta_i,\alpha_j\in\C^*,\beta_j=0$ for $i\in S^\prime,j\in S\setminus S^\prime$.  Assume   $\mathrm{Ind}(M)$ is  an  $\H$-module    defined by \eqref{ind2.2} for which $M$
satisfies the conditions in Theorem \ref{th2.1}. Then the tensor product $\big(\bigotimes_{i=1}^m\Omega(\lambda_i,\alpha_i,\beta_i)\big)\otimes\mathrm{Ind}(M)$ is an irreducible $\H$-module.
\end{theo}
\begin{proof}
 For any $w\in \mathrm{Ind}(M)$,  there exists $K(w)\in\Z_+$ such that $L_k\cdot w=I_k\cdot w=0$ for all $k\geq K(w)$ by Theorem \ref{th2.1}. Suppose $W$ is a nonzero submodule of $\big(\bigotimes_{i=1}^m\Omega(\lambda_i,\alpha_i,\beta_i)\big)\otimes\mathrm{Ind}(M)$.
Choose a nonzero element $u\in W$ with  minimal degree. We claim that $\mathrm{deg}(u)=0$. If  not, we assume
\begin{equation*}
u=\sum_{\mathbf{p}\in I} \partial_1^{p_1}\cdots\partial_m^{p_m}\otimes w_{\mathbf{p}}\in W,
\end{equation*}
where $I$ is a finite subset of $\Z_+^m$ and $w_{\mathbf{p}}$ are  nonzero vectors of $\mathrm{Ind}(M)$.
Let $\partial_1^{p_1}\cdots\partial_m^{p_m}\otimes w_{\mathbf{p}}$ be  maximal among the terms in the sum  with respect to $``\prec"$ and let $i^\prime$ be minimal such that   $p_{i^\prime}>0$.

First we consider $i^\prime\in S^\prime$.
For enough large $k\in\Z$, we obtain
\begin{equation}\label{4.2}
I_k\big(\sum_{\mathbf{p}\in I} \partial_1^{p_1}\cdots\partial_m^{p_m}\otimes w_{\mathbf{p}}\big)=\sum_{i=1}^m\sum_{\mathbf{p}\in I} \partial_1^{p_1}\cdots\lambda_i^k\beta_i(\partial_i-k)^{p_i}\cdots\partial_m^{p_m}\otimes w_{\mathbf{p}}\in W,
\end{equation}
where $I$ is a finite subset of $\Z_+^m$ and $w_{\mathbf{p}}$ are  nonzero vectors of $\mathrm{Ind}(M)$.
Now we consider $i^\prime\in S\setminus S^\prime$. For enough large $k\in\Z$, one can easily to get
\begin{equation}\label{4.1}
L_k\big(\sum_{\mathbf{p}\in I} \partial_1^{p_1}\cdots\partial_m^{p_m}\otimes w_{\mathbf{p}}\big)=\sum_{i=1}^m\sum_{\mathbf{p}\in I} \partial_1^{p_1}\cdots\lambda_i^k(\partial_i+k\alpha_i)(\partial_i-k)^{p_i}\cdots\partial_m^{p_m}\otimes w_{\mathbf{p}}\in W,
\end{equation}
where $I$ is a finite subset of $\Z_+^m$ and $w_{\mathbf{p}}$ are  nonzero vectors of $\mathrm{Ind}(M)$.

By Proposition \ref{pro2.3}, we respectively consider the coefficient of $\lambda_{i^\prime}^kk^{p_{i^\prime}}$ and $\lambda_{i^\prime}^kk^{p_{i^\prime}+1}$ in \eqref{4.2} and \eqref{4.1}, one has
\begin{equation*}
\partial_1^{{p}_1}\cdots\partial_{i^\prime-1}^{{p}_{i^\prime-1}}\partial_{i^\prime+1}^{{p}_{i^\prime+1}}\cdots\partial_m^{{p}_m}\otimes w_{{\mathbf{p}}}\in W,
\end{equation*}
where $m\in\N,w_{{\mathbf{p}}}$ are nonzero vectors of $\mathrm{Ind}(M)$.
Then
$$\partial_1^{{p}_1}\cdots\partial_{i^\prime-1}^{{p}_{i^\prime-1}}\partial_{i^\prime+1}^{{p}_{i^\prime+1}}\cdots\partial_m^{{p}_m}\otimes w_{{\mathbf{p}}}+\mathrm{lower\ terms}$$
has lower degree than $u$, which is contrary to the choice of $u$. Hence, $\mathrm{deg}(u)=0$.

By Lemma \ref{lemm431}, we see that $\big(\bigotimes_{i=1}^m\Omega(\lambda_i,\alpha_i,\beta_i)\big)\otimes w_{\mathbf{0}}$  can  be generated by  $\underbrace{1\otimes\cdots\otimes1}_{m}\otimes w_{\mathbf{0}}$.
It follows that $\big(\bigotimes_{i=1}^m\Omega(\lambda_i,\alpha_i,\beta_i)\big)\otimes \mathcal U(\H) w_\mathbf{0}\subseteq W.$ Thus,
 $W=\big(\bigotimes_{i=1}^m\Omega(\lambda_i,\alpha_i,\beta_i)\big)\otimes\mathrm{Ind}(M)$, since the nonzero $\H$-submodule $\mathcal U(\H) w_\mathbf{0}$ of ${\rm Ind}(M)$  generated by $w_\mathbf{0}$ is equal to ${\rm Ind}(M)$ by the irreducibility of ${\rm Ind}(M)$. This completes the proof of Theorem \ref{th1}.
\end{proof}
\begin{remark}
When $S^\prime={\O}$ in Theorem \ref{th1}, it was studied in \cite{TZ1}.
\end{remark}
It follows from Lemma \ref{lemm2} and Theorem \ref{th1} that we have the following remark.
\begin{remark}
Let $m\in\N,S=\{1,\ldots,m\},S^\prime\subseteq S$ and $\lambda_i\in\C^*$ for $i\in S$.
Let $\beta_i,\alpha_j\in\C^*,\beta_j=0$ for $i\in S^\prime,j\in S\setminus S^\prime$.  Assume   $\mathrm{Ind}(M)$ is  an  $\H$-module    defined by \eqref{ind2.2} for which $M$
satisfies the conditions in Theorem \ref{th2.1}. Then the tensor product $\big(\bigotimes_{i=1}^m\Omega(\lambda_i,\alpha_i,\beta_i)\big)\otimes\mathrm{Ind}(M)$ is an irreducible $\H$-module if and only if the
$\lambda_i$ pairwise distinct.
\end{remark}

\section{Isomorphism classes}
In this section, we will give isomorphism results for modules  $\big(\bigotimes_{i=1}^m\Omega(\lambda_i,\alpha_i,\beta_i)\big)\otimes \mathrm{Ind}(M)$.
We denote the number of elements in set $A$ by $\mathrm{card}(A)$.

\begin{theo}\label{th2}
Let   $m,n\in\N,S=\{1,\ldots,m\},T=\{1,\ldots,n\},S^\prime\subseteq S,T^\prime\subseteq T,\lambda_i,\mu_j\in\C^*$  with the
$\lambda_i$ pairwise distinct as well as $\mu_j$ pairwise distinct    for $i\in S,j\in T$.
Let $\beta_{i^\prime},\alpha_i\in\C^*,\beta_{i}=0$ and $d_{j^\prime},c_{j}\in\C^*,d_{j}=0$ for
$i^\prime\in S^\prime,{i}\in S\setminus S^\prime,j^\prime\in T^\prime,j\in T\setminus T^\prime$.
 Assume $\mathrm{Ind}(M_1)$ and $\mathrm{Ind}(M_2)$ are   $\H$-modules  defined by \eqref{ind2.2} for which $M_1$ and $M_2$ satisfy the conditions in Theorem $\ref{th2.1}$.  Then   $\big(\bigotimes_{i=1}^m\Omega(\lambda_i,\alpha_i,\beta_i)\big)\otimes \mathrm{Ind}(M_1)$ and   $\big(\bigotimes_{j=1}^n\Omega(\mu_j,c_j,d_j)\big)\otimes \mathrm{Ind}(M_2)$  are isomorphic
as $\H$-modules if and only if  $m=n,\mathrm{card}(S^\prime)=\mathrm{card}(T^\prime),\mathrm{Ind}(M_1)\cong \mathrm{Ind}(M_2)$
 as $\H$-modules and $(\lambda_{i},\alpha_{i},\beta_{i})=(\mu_{i^\prime},c_{i^\prime},d_{i^\prime})$  and $(\lambda_{j},\alpha_{j})=(\mu_{j^\prime},c_{j^\prime}),\beta_{j}=d_{j^\prime}=0$ for $i\in S^\prime,i^\prime\in T^\prime,j\in S\setminus S^\prime,j^\prime\in T\setminus T^\prime$ $\mathrm{(}$$\varphi_1:S^\prime\rightarrow T^\prime \ and\ \varphi_2:S\setminus S^\prime\rightarrow T\setminus T^\prime$ are both bijections $\mathrm{)}$.
\end{theo}
\begin{proof}
The sufficiency is trivial.
We  denote $\Omega(\lambda_i,\alpha_i,\beta_i)=\C[\partial_i],
\Omega(\mu_j,c_j,d_j)=\C[\widetilde{\partial}_j],V_1=\big(\bigotimes_{i=1}^{m}\Omega(\lambda_i,\alpha_i,\beta_i)\big)\otimes\mathrm{Ind}(M_1)$ and
$V_2=\big(\bigotimes_{j=1}^{n}\Omega(\mu_j,c_j,d_j)\big)\otimes \mathrm{Ind}(M_2)$, respectively.

 Let $\phi$ be an isomorphism from $V_1$ to   $V_2$.
Take a nonzero element  $\underbrace{1\otimes\cdots\otimes1}_{m}\otimes w\in V_1$. Assume
\begin{equation}\label{ass222}
\phi(\underbrace{1\otimes\cdots\otimes1}_{m}\otimes w)=\sum\limits_{\mathbf{p}\in I}\widetilde{\partial_1}^{p_1}\cdots\widetilde{\partial_n}^{p_n}\otimes v_{\mathbf{p}},
\end{equation}
where $I$ is a finite subset of $\Z_+^n$ and $v_{\mathbf{p}}$ are  nonzero vectors of $V_2$.
There exists  a positive  integer $K=\mathrm{max}\{K(w),K(v_{\mathbf{p}})\mid \mathbf{p}\in \Z_+^n\}$ such that $I_m\cdot w=I_m\cdot v_{\mathbf{p}}=L_m\cdot w=L_m\cdot v_{\mathbf{p}}=0$
for all integers $m\geq K$ and $\mathbf{p}\in \Z_+^n$.

First consider $d_{j^\prime}\in\C^*,d_j=0$ for
$j^\prime\in T^\prime,j\in T\setminus T^\prime$.
We note that $m=\mathrm{card}(S^\prime)+\mathrm{card}(S\setminus S^\prime),n=\mathrm{card}(T^\prime)+\mathrm{card}(T\setminus T^\prime)$.
For enough large  $k\in\Z$,    we know that
\begin{eqnarray}\label{Lk3.3}
\sum_{j=1}^m \lambda_j^k\beta_j\phi(\underbrace{1\otimes\cdots\otimes1}_{m}\otimes w)
&=&\phi\big(I_k(\underbrace{1\otimes\cdots\otimes1}_{m}\otimes w)\big)=
I_k\phi(\underbrace{1\otimes\cdots\otimes1}_{m}\otimes w)\nonumber
\\&=&
\sum_{j^\prime=1}^n\sum\limits_{\mathbf{p}\in \Z_+^n}\widetilde{\partial_1}^{p_1}\cdots\mu_{j^\prime}^kd_{j^\prime}(\widetilde{\partial_{j^\prime}}-k)^{p_{j^\prime}}\cdots\widetilde{\partial_n}^{p_n}\otimes v_{\mathbf{p}}.
\end{eqnarray}
According to Proposition \ref{pro2.3} in \eqref{Lk3.3}, we get $p_{j^\prime}=0,\lambda_j=\mu_{j^\prime},\mathrm{card}(T^\prime)=\mathrm{card}(S^\prime)$, where $j\in S^\prime,j^\prime\in T^\prime$ and $\varphi_1:S^\prime\rightarrow T^\prime$ is bijection.
Then $\phi(\underbrace{1\otimes\cdots\otimes1}_{m}\otimes w)=\sum\limits_{\widehat{\mathbf{p}}\in I}\widetilde{\partial_1}^{\widehat{p_1}}\cdots\widetilde{\partial_n}^{\widehat{p_n}}\otimes v_{\widehat{\mathbf{p}}}$,
where $\widehat{p_{j^\prime}}=0$ for $j^\prime\in T^\prime$.
Now we consider $c_j\in\C^*$ for
$j\in T\setminus T^\prime$.
For enough large $k\in\Z$, it follows from  $\phi\big(L_k(\underbrace{1\otimes\cdots\otimes1}_{m}\otimes w)\big)=L_k\phi(\underbrace{1\otimes\cdots\otimes1}_{n}\otimes w)$ that we have
\begin{eqnarray}\label{4444.4}
&&\sum_{i=1}^m \Big(\lambda_i^k\phi(1\otimes\cdots\otimes\partial_i\otimes\cdots\otimes1\otimes w)+\lambda_i^kk\alpha_i\phi(\underbrace{1\otimes\cdots\otimes1}_{m}\otimes w)\Big)\nonumber
\\&=&\sum_{j=1}^n\sum\limits_{\widehat{\mathbf{p}}\in \Z_+^n}\widetilde{\partial_1}^{\widehat{p_1}}\cdots\mu_j^k(\widetilde{\partial_j}+kc_j)
(\widetilde{\partial_j}-k)^{\widehat{p_j}}\cdots\widetilde{\partial_n}^{\widehat{p_n}}\otimes v_{\widehat{\mathbf{p}}}.
\end{eqnarray}
Using Proposition \ref{pro2.3} in \eqref{4444.4}, one can easily  to check that  $\widehat{p_j}=0,\lambda_i=\mu_j,\mathrm{card}(T\setminus T^\prime)=\mathrm{card}(S\setminus S^\prime)$, where $i\in S\setminus S^\prime,j\in T\setminus T^\prime$  and $\varphi_2:S\setminus S^\prime\rightarrow T\setminus T^\prime$ is bijection. Then $m=\mathrm{card}(T\setminus T^\prime)+\mathrm{card}(T)=\mathrm{card}(S\setminus S^\prime)+\mathrm{card}(S)=n$
and
\begin{eqnarray*}
\phi\big(\underbrace{1\otimes\cdots\otimes1}_{m}\otimes w\big)
=\underbrace{1\otimes\cdots\otimes1}_{m}\otimes v_{\mathbf{0}}
\end{eqnarray*}
 Thus, \eqref{Lk3.3} can be rewritten   as
\begin{eqnarray*}
\sum_{i=1}^m \lambda_i^k\beta_i\phi\big(\underbrace{1\otimes\cdots\otimes1}_{m}\otimes w\big)
=\sum_{i=1}^m\mu_i^kd_i\big(\underbrace{1\otimes\cdots\otimes1}_{m}\otimes v_{\mathbf{0}}\big),
\end{eqnarray*}
we obtain
$\beta_i=d_{i^\prime}$ for $i\in S^\prime,i^\prime\in T^\prime$, which can be obtained by $\phi\big(\underbrace{1\otimes\cdots\otimes1}_{m}\otimes w\big)\neq0,\lambda_i=\mu_{i^\prime}$ and Proposition \ref{pro2.3}. We note that
$\beta_j=d_{j^\prime}=0$
for $j\in S\setminus S^\prime,j^\prime\in T\setminus T^\prime$. Then   for enough large $k\in\Z$, by   $\phi\big(L_k(\underbrace{1\otimes\cdots\otimes1}_{m}\otimes w)\big)=L_k(\underbrace{1\otimes\cdots\otimes1}_{m}\otimes v_{\mathbf{0}})$ and $\lambda_i=\mu_{i^\prime},\lambda_j=\mu_{j^\prime}$,
 we can easily  check that $\alpha_i=c_{i^\prime},\alpha_j=c_{j^\prime}$  for $i\in S^\prime,i^\prime\in T^\prime, j\in S\setminus S^\prime,j^\prime\in T\setminus T^\prime$.

 There exists a linear bijection $\tau:V_1\rightarrow V_2$
such that
$$\phi(\underbrace{1\otimes\cdots\otimes1}_{m}\otimes v)=\underbrace{1\otimes\cdots\otimes1}_{m}\otimes \tau(v)$$ for all $v\in V_1.$  Meanwhile we  conclude  that
$\tau(L_k v)=L_k \tau(v)$ for all $k\in\Z,v\in V_1.$
Then  from $$\phi\big(I_k(\underbrace{1\otimes\cdots\otimes1}_{m}\otimes v)\big)=I_k\phi(\underbrace{1\otimes\cdots\otimes1}_{m}\otimes v)$$ and $$\phi\big(C_i(\underbrace{1\otimes\cdots\otimes1}_{m}\otimes v)\big)=C_i\phi(\underbrace{1\otimes\cdots\otimes1}_{m}\otimes v),$$
we   see that   $\tau(I_k v)=I_k \tau(v)$ and $\tau(C_i v)=C_i \tau(v)$  for $i=1,2,3, k\in\Z$, respectively.
 Thus, $V_1\cong V_2$ as $\H$-modules for $d_{j^\prime},c_j\in\C^*,d_j=0$ for
$j^\prime\in T^\prime,j\in T\setminus T^\prime$.

To sum up,
 we   obtain    $m=n,V_1\cong V_2,(\lambda_{i},\alpha_{i},\beta_{i})=(\mu_{i^\prime},c_{i^\prime},d_{i^\prime}),(\lambda_{j},\alpha_{j})=(\mu_{j^\prime},c_{j^\prime})$ and $\beta_{j}=d_{j^\prime}=0$ for $i\in S^\prime,i^\prime\in T^\prime,j\in S\setminus S^\prime,j^\prime\in T\setminus T^\prime$. This completes the proof.
\end{proof}
\begin{remark}
When $S^\prime=T^\prime={\O}$ in Theorem \ref{th2}, it was investigated in   \cite{TZ1}.
\end{remark}

\section{New irreducible modules}
In this section, we shall  show that   $\big(\bigotimes_{i=1}^m\Omega(\lambda_i,\alpha_i,\beta_i)\big)\otimes\mathrm{Ind}(M)$ is not isomorphic to  $\mathrm{Ind}(M)$ or  $\mathcal{M}\big(V,\Omega(\lambda,\alpha,\beta)\big)$ or
$\widetilde{\mathcal M}(W,\gamma(t))$ or $\A_{\alpha,\beta}$.

 For any $l,m\in\Z$, $s\in\Z_+$,  define a sequence of  operators $T_{l,m}^{(s)}$ as follows
\begin{equation*}
T_{l,m}^{(s)}=\sum_{i=0}^s(-1)^{s-i} \binom{s}{i}I_{l-m-i}I_{m+i}.
\end{equation*}
 For $d\in\{0,1\},r\in\Z_+$,
  denote by $\H_{r,d}$
the Lie subalgebra of $\H_{+,d}=\mathrm{span}_{\C}\{L_i,I_j\mid i\geq0,j\geq d\}$  generated by  $L_i,I_j$ for all $i>r,j>r+d$.
 Now we write    $\bar \H_{r,d}$
 the quotient algebra $\H_{+,d}/ \H_{r,d}$,
 and   $\bar L_i,\bar I_{i+d}$ the respective images of $L_i,I_{i+d}$
in $\bar \H_{r,d}$.

 Let $d\in\{0,1\}, r\in\Z_+$ and $V$  be an $\bar\H_{r,d}$-module.
For any  fixed $\gamma(t)=\sum_{i}c_it^i\in\C[t,t^{-1}]$,
define the  action of $\H$ on $V\otimes \C[t,t^{-1}]$ as follows
\begin{eqnarray*}\label{l6.1}
&&L_m \circ(v\otimes t^n)=(L_m+\sum_{i}c_iI_{m+i})(v\otimes t^n),
\\&&I_m\circ(v\otimes t^n)=I_m(v\otimes t^n),\\
&&C_i\circ(v\otimes t^n)=0\quad {\rm for}\ m,n\in\Z, v\in V\ {\rm and}\ i=1,2,3.
 \end{eqnarray*}
Then $V\otimes \C[t,t^{-1}]$ carries the structure of an $\H$-module under the  above given actions, which is denoted by $\widetilde{\mathcal M}(V,\gamma(t))$. Note that  $\widetilde{\mathcal M}(V,\gamma(t))$ is a weight $\H$-module if and only if $\gamma(t)\in \C$ and also that
the $\H$-module $\widetilde{\mathcal M}(V,\gamma(t))$ for $\gamma(t)\in\C[t,t^{-1}]$ is  irreducible if and only if $V$ is irreducible (see \cite{CHS16}).

 Let $d\in\{0,1\}, r\in\Z_+$ and  $V$ be an $\bar \H_{r,d}$-module.
For any $\lambda, \alpha,\beta\in\C$, define an $\H$-action  on the vector space $\mathcal{M}\big(V,\Omega(\lambda,\alpha,\beta)\big):=V\otimes \C [t]$ as follows\vspace{-0.3cm}
 \begin{eqnarray*}
 && L_m\big(v\otimes f(t)\big)=v\otimes\lambda^m(t-m\alpha)f(t-m)+\sum_{i=0}^r\Big(\frac{m^{i+1}}{(i+1)!}\bar L_i\Big)v\otimes \lambda^mf(t-m),\\&&
I_m\big(v\otimes f(t)\big)=\sum_{i=0}^r\Big(\frac{m^{i+d}}{(i+d)!}\bar I_{i+d}\Big)v\otimes  \lambda^m\beta f(t-m),\\&&
C_i\big(v\otimes f(t)\big)=0\quad {\rm for}\ i\in\{1,2,3\}, m\in\Z,v\in V, f(t)\in\C[t].
 \end{eqnarray*}
We note that  $\mathcal{M}\big(V,\Omega(\lambda,\alpha,\beta)\big)$ is  reducible if and only if  $V\cong V_{\alpha,\delta_{d,0}\tau}$ for some $\tau\in\C$ such that $\delta_{d,0}\beta\tau=0$ (see \cite{CHSY}).

\begin{lemm}\label{55.1}
Let $ \lambda_i\in\C^*, \alpha,\beta,\alpha_i,\beta_i\in\C,s\in\Z_+,d\in\{0,1\}$ and  $M$ be an irreducible $\H_e$-module satisfying the conditions in Theorem $\ref{th2.1}$. Assume that $r^\prime$ is the maximal nonnegative integer such that $\bar I_{r^\prime+d}V\neq0$.  Then
\begin{itemize}\lineskip0pt\parskip-1pt
\item[\rm (i)] the action of $L_m$ for $m$ sufficiently large is not locally nilpotent on $\big(\bigotimes_{i=1}^m\Omega(\lambda_i,\alpha_i,\beta_i)\big)\otimes\mathrm{Ind}(M);$
\item[\rm (ii)]    the action of $T_{l,m}^{(s)}$  on $\A_{\alpha,\beta}$ is trivial for $l,m\in\Z$;
\item[\rm (iii)] $T_{l,m}^{(1)}$ acts nontrivially on $\big(\bigotimes_{i=1}^m\Omega(\lambda_i,\alpha_i,\beta_i)\big)\otimes\mathrm{Ind}(M)$ whenever  $m\ll 0$ and $l\ll m$;
   \item[\rm (iv)]    The action of  $T_{l,m}^{(s)}$  on  $\mathcal{M}\big(V,\Omega(\lambda,\alpha,\beta)\big)$ and $\widetilde{\mathcal M}(V,\gamma(t))$ are
  trivial for $s>2(r^\prime+d)$.
 \end{itemize}
\end{lemm}
\begin{proof} (i) follows from the local nilpotency  of $L_m$   on $\mathrm{Ind}(M)$ by Theorem \ref{th2.1} for $m$ sufficiently large and its  non-local nilpotency on $\Omega(\lambda,\alpha,\beta)$.
 (ii) follows from \eqref{LI2.1}.
(iii) can be obtained by the similar compute in Lemma 5.1 (v) of \cite{CHSY}.
  {\rm (iv)}  follows from   \cite[Lemma 3.3]{CHS16}.
\end{proof}
We are now ready to state the main result of this section.
\begin{prop}
Let $d\in\{0,1\}$, $r,e\in\Z_+,$   $\alpha,\beta\in\C,$ $M$ be an irreducible $\H_e$-module satisfying the conditions in  Theorem $\ref{th2.1}$. Then
$\big(\bigotimes_{i=1}^m\Omega(\lambda_i,\alpha_i,\beta_i)\big)\otimes\mathrm{Ind}(M)$
is not isomorphic to  $\mathrm{Ind}(M^\prime)$ for any simple $\H_e$-module $M^\prime$ satisfying the conditions in Theorem $\ref{th2.1}$, or $\mathcal{M}\big(V,\Omega(\lambda,\alpha^\prime,\beta^\prime)\big)$ for any $\bar \H_{r,d}$-module $V$, $\lambda\in\C^*,\alpha^\prime,\beta^\prime\in\C$,   or $\widetilde{\mathcal M}(W,\gamma(t))$ for any $\bar\H_{r,d}$-module $W$ and $\gamma(t)=\sum_{i}c_it^i\in\C[t,t^{-1}]$, or $\A_{\alpha,\beta}$ for $\alpha,\beta\in\C$.
\end{prop}
\begin{proof}
From Lemma \ref{55.1} (i),  we have $\big(\bigotimes_{i=1}^m\Omega(\lambda_i,\alpha_i,\beta_i)\big)\otimes\mathrm{Ind}(M)\ncong\mathrm{Ind}(M^\prime)$.
Let $m\ll 0,l\ll m$ that  $I_{l-m},I_{m}\notin\H_e$ and $s>2(r^\prime+d)$. For any $1\otimes v\in\big(\bigotimes_{i=1}^m\Omega(\lambda_i,\alpha_i,\beta_i)\big)\otimes\mathrm{Ind}(M)$,
 noting that the action of $I_m$ on $1$ is scalar for any $m\in\Z$, we deduce that
\begin{eqnarray*}
&&T_{l,m}^{(s)}(1\otimes v)
\\&=&\sum_{i=0}^s(-1)^{s-i} \binom{s}{i}\big(I_{l-m-i}I_{m+i}(1)\otimes v+I_{m+i}(1)\otimes I_{l-m-i}v\\&&
+I_{l-m-i}(1)\otimes I_{m+i}v+1\otimes I_{l-m-i}I_{m+i}v\big)\neq0.
\end{eqnarray*}
Then by Lemma \ref{55.1} (iv),
we obtain that $\big(\bigotimes_{i=1}^m\Omega(\lambda_i,\alpha_i,\beta_i)\big)\otimes\mathrm{Ind}(M)$
is not isomorphic to  $\mathcal{M}\big(V,\Omega(\lambda,\alpha,\beta)\big)$,   or $\widetilde{\mathcal M}(W,\gamma(t))$.

At last, by  Lemma \ref{55.1} (ii) and (iii), we get that $\big(\bigotimes_{i=1}^m\Omega(\lambda_i,\alpha_i,\beta_i)\big)\otimes\mathrm{Ind}(M)$ is not isomorphic to $\A_{\alpha,\beta}$.
\end{proof}

\section*{Acknowledgements}

This work was  partially
supported by  the NSFC (11801369, 11431010).
We would like to thank Prof.
Jianzhi Han for his useful   discussions.

\bigskip

\small \begin{thebibliography}{9999}\vskip0pt
\parindent=2ex\parskip=-1pt\baselineskip=-1pt
\bibitem{ADK}  E. Arbarello, C. DeConcini, V. G.  Kac,  C. Procesi,  Moduli spaces of curves and representation theory, {\it Commun.
Math. Phys.} {\bf 117}  (1988), 1-36.


\bibitem{CG} H. Chen, X. Guo, New simple modules for the Heisenberg-Virasoro algebra, {\it J. Algebra}  {\bf 390} (2013), 77-86.


\bibitem{CG2} H. Chen, X. Guo, Non-weight modules over the Heisenberg-Virasoro
algebra and the $W$ algebra $W(2,2)$, {\it J. Algebra Appl.} {\bf 16} (2017), 1750097.


\bibitem{CHS16}	H. Chen, J. Han, Y. Su, A class of simple weight modules over the twisted Heisenberg-Virasoro algebra, {\it J. Math. Phys.}  {\bf57}  (2016),   101705, 7 pp.

\bibitem{CHSY}	H. Chen, J. Han,  Y. Su, X. Yue, Two classes of non-weight modules over the twisted Heisenberg-Virasoro algebra,  {\it Manuscripta Mathematica} DOI:10.1007/s00229-018-1059-3.




\bibitem{HCS} J. Han, Q. Chen, Y. Su, Modules over the algebra $\mathcal V ir(a, b)$,  {\it Linear Algebra Appl.} {\bf 515} (2017), 11-23.



\bibitem{KS} I. Kaplansky, L. J. Santharoubane, Harish-Chandra modules over the Virasoro algebra, Math. Sci. Res. Inst. Publ. {\bf4} (1985), Springer, New York,  217-231.

\bibitem{LGZ} R. L\"{u}, X. Guo, K. Zhao, Irreducible modules over the Virasoro algebra,
{\it Doc. Math.} {\bf 16} (2011), 709-721.

\bibitem{LJ}
D. Liu, C. Jiang, Harish-Chandra modules over the twisted Heisenberg-Virasoro algebra, {\it J. Math. Phys.} {\bf 49} (2008),   012901, 13 pp.

\bibitem{LWZ} D. Liu, Y. Wu, L. Zhu, Whittaker modules for the twisted Heisenberg-Virasoro algebra, {\it J. Math. Phys.} {\bf 51}  (2010),  023524, 12 pp.


\bibitem{LZK} R. L\"{u}, K. Zhao, Irreducible Virasoro modules from irreducible Weyl modules, {\it J. Algebra}  {\bf 414} (2014), 271-287.



\bibitem{LZ1} R. L\"{u}, K. Zhao, Classification of irreducible weight modules over the
twisted Heisenberg-Virasoro algebra, {\it Commun. Contemp. Math.}
{\bf 12} (2010), 183-205.




\bibitem{LZ0} R. L\"{u}, K. Zhao,  A family of simple weight Virasoro modules, {\it J. Algebra}  {\bf 479} (2017), 437-460.







\bibitem{R} G. Radobolja, Subsingular vectors in Verma modules, and tensor product of weight modules  over the twisted Heisenberg-Virasoro algebra and $W(2,2)$ algebra,
 {\it J. Math. Phys.} {\bf 54} (2013),  071701, 24 pp.





\bibitem{SS07}	R. Shen, Y. Su, Classification of irreducible weight modules with a finite-dimensional weight space over twisted Heisenberg-Virasoro algebra, {\it Acta Math. Sin. $\mathrm{(}$E. S.$\mathrm{)}$} {\bf 23} (2007), 189-192.


\bibitem{TZ1} H. Tan, K. Zhao, Irreducible Virasoro modules from tensor products $(\mathrm{II})$, {\it J. Algebra}
{\bf 394} (2013), 357-373.

\end{thebibliography}
\end{document}